\documentclass[a4paper,11pt,reqno]{amsart}
\usepackage{amsmath,amsthm,amssymb}
\usepackage[numbers,sort&compress]{natbib}

\newtheorem{theorem}{Theorem}
\nonstopmode \numberwithin{equation}{section}
\theoremstyle{definition}

\begin{document}

\title[ generalized Bessel-Maitland function ]{Unified integral operator involving generalized Bessel-Maitland function}


\author[W. A. Khan]{Waseem Ahmad Khan}
\address{Department of Mathematics\\ 
 Integral University, Lucknow-226026\\
India}
\email{ waseem08\_khan@rediffmail.com}
\author[K.S. Nisar]{Kottakkaran Sooppy Nisar*}
\address{epartment of Mathematics, College of Arts and Science at Wadi Aldawaser\\ 
11991, Prince Sattam bin Abdulaziz University, Alkharj\\
Kingdom of Saudi Arabia}
\email{ksnisar1@gmail.com}

%
\thanks{*Corresponding Author}

\maketitle

\begin{abstract}
The main object of this article is to present an interesting double integral involving generalized Bessel-Maitland function defined by Ghaysuddin et al. \cite{9}, which is expressed in terms of generalized (Wright) hypergeometric function. We also considered some special cases as an application of the main result.
\end{abstract}

\section{First section}

Recently many authors namely, Ali \cite{1}, Choi and Agarwal \cite{3,4,5}, Khan et al. \cite{10,11,12}, Ghayasuddin et al. \cite{8,9}, Abouzaid et al. \cite{2} have introduced integral formulae associated with Bessel function. It has a wide application in the problem of Physics, Chemistry, Biology, Engineering and Applied Sciences. The theory of Bessel functions is closely associated with the theory of certain types of differential equations. An elaborate account of applications of Bessel functions are given in the book of Watson \cite{24}.

The Bessel-Maitland function $J_{\nu}^{\mu}(z)$ \cite{14} is defined by the following series representation:
$$J_{\nu}^{\mu} (z)=\sum\limits_{n=0}^{\infty}\frac{(-z)^n}{n!\Gamma(\mu n+\nu+1)}=\phi(\mu,\nu+1;-z),\eqno(1.1)$$

Currently, Singh et al. \cite{23} introduced the following generalization of Bessel-Maitland function:
$$J_{\nu,q}^{\mu,\gamma}=\sum\limits_{n=0}^\infty\frac{(\gamma)_{qn} (-z)^n}{\Gamma(\mu n +\nu+1)n!},\eqno(1.2)$$

where $\mu,\nu,\gamma \in \mathbb C, Re(\mu)\geq 0, Re(\nu)\geq-1, Re(\gamma)\geq 0$ and $q\in (0,1)\bigcup \mathbb N$ and $(\gamma)_{0}=1, (\gamma)_{qn}=\frac{\Gamma(\gamma+qn)}{\Gamma(\gamma)}$ denotes the generalized Pochhammer symbol.

With reference to the work mentioned above, Ghayasuddin et al. \cite{9} introduced and investigated a new extension of Bessel-Maitland function as follows:

$$J_{\nu,\gamma,\delta}^{\mu,q,p}(z)=\sum\limits_{n=0}^{\infty}\frac{(\gamma)_{qn}(-z)^n}{\Gamma(\mu n+\nu+1)(\delta)_{pn}},\eqno(1.3)$$

where $\mu,\nu,\gamma,\delta\in \mathbb C$, $\Re(\mu)\geq 0, \Re(\nu)\geq -1, \Re(\gamma)\geq 0, \Re(\delta)\geq 0; p,q>0$ and $q<\Re(\alpha)+p.$

The Mittag-Leffler (Swedish Mathematician) \cite{15} introduced the function $E_{\alpha}(z)$ and defined it as:
$$ E_{\alpha}(z) = \sum\limits_{n=0}^{\infty}\frac{z^{n}}{\Gamma(\alpha n +1)}, \eqno(1.4)$$
where $z\in \mathbb C$ and ${\Gamma (s)}$ is the Gamma function; ${\alpha \geq0}$.

The Mittag-Leffler function is a direct generalization of $\exp(z)$ in which $\alpha =1$. Mittag-Leffler function naturally occurs as the solution of fractional order differential equation or fractional order integral equations.

A generalization of $E_{\alpha}(z)$ was studied by Wiman \cite{25} where he defined the function $E_{\alpha,\beta}(z)$ as:
$$E_{\alpha,\beta}(z)= \sum\limits_{n=0}^{\infty}\dfrac{z^{n}}{\Gamma(\alpha n +\beta)}, \eqno(1.5)$$
where $\alpha, \beta\in\mathbb C; \Re(\alpha)>0, \Re(\beta)>0$ which is also known as Mittag-Leffler function or Wiman's function.

In (1903), Prabhakar \cite{16} introduced the function $E_{\alpha,\beta}^{\gamma}(z)$ in the form (see Killbas et al. \cite{13}) as:

$$E_{\alpha,\beta}^{\gamma}(z)= \sum\limits_{n=0}^{\infty}\dfrac{{(\gamma)_{n}}}{\Gamma(\alpha n +\beta)}\dfrac{z^{n}}{n!}, \eqno(1.6)$$

where $\alpha, \beta, \gamma\in\mathbb C  ; \Re(\alpha)>0, \Re(\beta)>0, \Re(\gamma)>0.$ 

Shukla and Prajapati \cite{19} defined and investigated the function $E_{\alpha,\beta}^{\gamma,q}(z)$ as:

$$E_{\alpha,\beta}^{\gamma,q}(z)= \sum\limits_{n=0}^{\infty}\dfrac{{(\gamma)_{qn}}}{\Gamma(\alpha n +\beta)}\dfrac{z^{n}}{n!}, \eqno(1.7)$$

where $\alpha, \beta, \gamma\in\mathbb C  ; \Re(\alpha)>0, \Re(\beta)>0, \Re(\gamma)>0, q \in (0,1)\bigcup \mathbb N $
and $(\gamma)_{qn} = \dfrac{\Gamma(\gamma +qn)}{\Gamma(\gamma)}$ denotes the generalized Pochhammer symbol which in particular reduces to $q^{qn}\prod\limits_{r=1}^{q}(\dfrac{\gamma +r-1}{q})_{n}$ if $q\in \mathbb N.$

Salim \cite{22} introduced a new generalized Mittag-Leffler function and defined it as:
$$E_{\alpha,\beta}^{\gamma,\delta}(z)= \sum\limits_{n=0}^{\infty}\dfrac{{(\gamma)_{n}}}{\Gamma(\alpha n +\beta)}\dfrac{z^{n}}{(\delta)_{n}},\eqno(1.8)$$
where $\alpha, \beta, \gamma, \delta\in\mathbb C; \Re(\alpha)>0, \Re(\beta)>0, \Re(\gamma)> 0, \Re(\delta)> 0.$ 

Afterward, Salim and Faraj \cite{21} introduced the generalized Mittag-Leffler function$E_{\alpha,\beta, p}^{\gamma,\delta,q}(z)$ which is defined as:
$$E_{\alpha,\beta, p}^{\gamma,\delta,q}(z)= \sum\limits_{n=0}^{\infty}\dfrac{{(\gamma)_{qn}}}{\Gamma(\alpha n +\beta)}\dfrac{z^{n}}{(\delta)_{pn}}, \eqno(1.9)$$
where ${\alpha,\beta,\gamma,\delta} \in \mathbb C, \min\{\Re(\alpha)>0, \Re(\beta)>0, \Re(\gamma)> 0, \Re(\delta)> 0\}>0; p,q>0$ and  ${q < \Re\alpha+p}.$

The generalization of the generalized hypergeometric series
${}_pF_{q}$ is due to Fox \cite{7} and Wright (\cite{26}, \cite{27}, \cite{28}) who studied the asymptotic expansion of the generalized (Wright) hypergeometric function defined by (see [p.21]\cite{20}; see also \cite{18}):
$${}_p\Psi_{q}\left[\begin{array}{ccc}
(\alpha_{1}~,~A_{1}), &.....,& (\alpha_{p}~,~A_{p});  \\
~ & ~ &~ \\
(\beta_{1}~,~B_{1}), &.....,& (\beta_{q}~,~B_{q}); \\
\end{array}~z\right]~= \sum_{k=0}^{\infty} \frac{\prod\limits_{j=1}^{p}~\Gamma{(\alpha_{j}+A_{j}k)}}{\prod\limits_{j=1}^{q}~\Gamma{(\beta_{j}+B_{j}k)}}~\frac{z^{k}}{k!}~,~~~~~~~~~\eqno(1.10)$$
where the coefficients $A_{1},\cdots, A_{p}$ and $B_{1},\cdots,
B_{q}$ are positive real numbers such that
$${\rm(i)}~~1+\sum_{j=1}^{q} B_{j}-\sum_{j=1}^{p} A_{j}>0~{\rm and}~0<|z|<\infty;~z\neq 0.~~~~~~~~~~~~~~~~~~~~~~~~~~~~~~~~~~~~~~~~~~~~~~~~~~~~~~~~~~~~\eqno(1.11)$$
$${\rm(ii)}~~1+\sum_{j=1}^{q} B_{j}-\sum_{j=1}^{p} A_{j}=0~{\rm and}~0<|z|<{A_{1}}^{-A_{1}}\dots {A_{p}}^{-A_{p}}{B_{1}}^{B_{1}}\dots {B_{q}}^{B_{q}}.~~~~~~~~~~~~~~~~~~~~~\eqno{(1.12)}$$

\noindent A special case of (1.10) is
$${}_p\Psi_{q}\left[\begin{array}{ccc}
(\alpha_{1}~,~1), &.....,& (\alpha_{p}~,~1);  \\
~ & ~& ~\\
(\beta_{1}~,~1), &.....,& (\beta_{q}~,~1); \\
\end{array}~z\right]~=~\frac{\prod\limits_{j=1}^{p}~\Gamma{(\alpha_{j})}}{\prod\limits_{j=1}^{q}~\Gamma{(\beta_{j})}}~{}_pF_{q}\left[\begin{array}{ccc}
\alpha_{1}, &.....,& \alpha_{p}~;  \\
~ & ~&~ \\
\beta_{1}, &.....,& \beta_{q}~; \\
\end{array}~z\right],~~~\eqno(1.13)$$
where ${}_pF_{q}$ is the generalized hypergeometric series defined by \cite{17}:
$${}_pF_{q}\left[\begin{array}{ccc}
\alpha_{1}, &.....,& \alpha_{p}~;  \\
~ & ~&~ \\
\beta_{1}, &.....,& \beta_{q}~; \\
\end{array}~z\right]=~\sum_{n=0}^{\infty} \frac{(\alpha_{1})_{n}\cdots (\alpha_{p})_{n}}{(\beta_{1})_{n}\cdots (\beta_{q})_{n}}~\frac{z^{n}}{n!}$$
$$~~~~~~~~~~~~~~~~~~~~~~~~~~~~~~~~~~~~~~~~~~~~~~~~~~={}_pF_{q}(\alpha_{1},\cdots,\alpha_{p};~\beta_{1},\cdots\beta_{q};~z),~~\eqno(1.14)$$
where $(\lambda)_{n}$ is the Pochhammer's symbol \cite{17}.

Furthermore, we also state here the following interesting and useful result stated by Edward [p.445]\cite{6}:
$$\int_{0}^{1}\int_{0}^{1}y^{\lambda}(1-x)^{\lambda-1}(1-y)^{\mu-1}(1-xy)^{1-\lambda-\mu}dx= \frac{\Gamma(\lambda)\Gamma(\mu)}{\Gamma(\lambda+\mu)},\eqno(1.15)$$
provided $0<\Re(\mu)<\Re(\lambda).$

\section{Integral operator involving generalized Bessel-Maitland function}

\begin{theorem}\label{Th2.1}
If $\nu,\delta,\eta,\gamma,\lambda\in \mathbb C$, $a\neq 0$, $\eta+p-q>0$, $\Re(\lambda)\geq0$, $\Re(\nu)\geq-1,\Re(\delta)\geq0,\Re(\gamma)\geq0$; $p,q>0$ and $q<\Re(\alpha)+p$, then  $$\int_{0}^{1}\int_{0}^{1}y^{\lambda}(1-x)^{\lambda-1}(1-y)^{\mu-1}(1-xy)^{1-\lambda-\mu}J_{\nu,\delta,p}^{\eta,\gamma,q}\left[\frac{ay(1-x)(1-y)}{(1-xy)^2}\right]dxdy$$
$$=\frac{\Gamma(\delta)}{\Gamma(\gamma)}{}_4\Psi_{3}\left[\begin{array}{lll}
(\lambda, 1),(\mu,1),(\gamma,q),(1,1);\\
&-a\\
(\nu+1,\eta),(\delta,p),(\lambda+\mu,2);\end{array}\right].\eqno(2.1)$$

\end{theorem}

\begin{proof}
 To establish our main result (2.1), we denote the left-hand side of (2.1) by $I$ and then by using (1.9), we have:
$$I=\int_{0}^{1}\int_{0}^{1}y^{\lambda}(1-x)^{\lambda-1}(1-y)^{\mu-1}(1-xy)^{1-\lambda-\mu}$$
$$\times \sum\limits_{n=0}^{\infty}\frac{(\gamma)_{qn}}{{\Gamma(\eta n+\nu+1)}(\delta)_{pn}}\left[\frac{-ay(1-x)(1-y)}{(1-xy)^2}\right]^{n}dxdy.\eqno(2.2)$$

Now changing the order of integration and summation, (which is guaranteed under the given condition), to get:

$$=\frac{\Gamma(\delta)}{\Gamma(\gamma)}\sum\limits_{n=0}^{\infty}\frac{\Gamma(\gamma+qn)\Gamma(\lambda+n)\Gamma(\mu+n)\Gamma(1+n)}{\Gamma(\eta n+\nu+1)\Gamma(\delta+pn)\Gamma(\lambda+\mu+2n)}\frac{a^n}{n!}.\eqno(2.3)$$
Finally, summing up the above series with the help of (1.10 ), we easily arrive at the right-hand side of (2.3). This completes the proof of our main result.
\end{proof}

Next, we consider other variation of (2.1). In fact, we establish an integral formula for the generalized Bessel-Maitland function $J_{\nu,\delta,p}^{\eta,\gamma,q}(z)$, which is expressed in terms of the generalized hypergeometric function ${}_pF_{q}.$

\vspace{.35cm} \noindent {\bf  Variation of (2.1):} Let the conditions of our main result be satisfied, then the following integral formula holds true:
$$\int_{0}^{1} \int_{0}^{1} y^{\lambda}~(1-x)^{\lambda-1}~(1-y)^{\mu-1}~(1-xy)^{1-\lambda-\mu}~J_{\nu,\delta,p}^{\eta,\gamma,q} \left [\frac{ay(1-x)(1-y)}{(1-xy)^{2}}\right]~{dx}~{dy}~~~~~~~~~~~~~$$

$$=\frac{\Gamma{(\lambda)}\Gamma{(\mu)}}{\Gamma(\nu+1)\Gamma{(\lambda+\mu)}}$$

$${}_{q+3}F_{\eta+p+2}\left[
\begin{array}{ccccc}
~\Delta(q;~\gamma), & ~\lambda, & ~\mu, & ~1, & ~; \\
~ & ~ & ~ & ~ & ~ \\
~\Delta(\eta;\nu+1), & ~\Delta(p;\delta) ,& ~\Delta(2;\lambda+\mu); & ~ \\
\end{array}\frac{-a q^q}{4\eta^{\eta}p^p}
\right],\\\eqno(2.4)$$

where $\Delta(m;~l)$ abbreviates the array of m parameters $\frac{l}{m},~\frac{l+1}{m},~\cdots,\frac{l+m-1}{m}~,$\\

$~m\geq1.$

\begin{proof}
 In order to prove the result (3.1), using the results $$\Gamma (\alpha+n)=\Gamma(\alpha)(\alpha)_n$$
and
$$(l)_{kn}=~k^{kn}\left(\frac{l}{k}\right)_{n}\left(\frac{l+1}{k}\right)_{n} \cdots
\left(\frac{l+k-1}{k}\right)_{n},$$
(Gauss multiplication theorem) in (2.3) and summing up the given series with the help of (1.14), we easily arrive at our required result (2.4).
\end{proof}

\section{ Special Cases}

\noindent {\bf (i).} On replacing $ \nu$ by $\nu-1 $ in (2.1) and then by using (1.9), we get:
$$\int_{0}^{1} \int_{0}^{1} y^{\lambda}~(1-x)^{\lambda-1}~(1-y)^{\mu-1}~(1-xy)^{1-\lambda-\mu}~E_{\nu,\delta,p}^{\eta,\gamma,q} \left[\frac{ay(1-x)(1-y)}{(1-xy)^{2}}\right]~{dx}~{dy}~~~~~~~~~~~~~$$
$$=~\frac{\Gamma(\delta)}{\Gamma(\gamma)}~{}_4\Psi_{3}\left[\begin{array}{ccccc}
                                                 (\gamma,~q), & (\lambda,~1), & (\mu,~1) & (1,~1) ;&  \\
                                                  &  &  &  & ~a \\
                                                 (\eta~,~\nu), & (\delta,~p), & (\lambda+\mu,2) & ; &
                                               \end{array}\right],~~\eqno(3.1)$$

where $\nu,\mu,\gamma,\delta,\lambda \in\mathbb C$, $\Re(\nu)>0, \Re(\lambda)>0, \Re(\gamma)>0, \Re(\delta)>0$; $p,q>0$ and $q<\Re(\alpha)+p$.

\vspace{.35cm}
\noindent {\bf (ii).} On replacing $ \nu$ by $\nu-1 $ in (2.4) and then by using (1.9), we found:
$$\int_{0}^{1} \int_{0}^{1} y^{\lambda}~(1-x)^{\lambda-1}~(1-y)^{\mu-1}~(1-xy)^{1-\lambda-\mu}~E_{\nu,\delta,p}^{\eta,\gamma,q} \left[\frac{ay(1-x)(1-y)}{(1-xy)^{2}}\right]~{dx}~{dy}~~~~~~~~~~~~~$$
$$=\frac{\Gamma(\lambda)\Gamma(\mu)}{\Gamma(\nu)\Gamma(\lambda+\mu)}{}_{q+3}F_{\eta+p+2}\left[
                                                                                                \begin{array}{ccccc}
                                                                                                  ~\Delta(q;\gamma), & ~\lambda ,& ~\mu ,& ~1 & ; \\
                                                                                                  ~ & ~ & ~ & ~ & ~ \\
                                                                                                  ~\Delta(\eta;\nu), & ~\Delta(p;\delta), & ~\Delta(2;\lambda+\mu), & ~ & ; \\
                                                                                                \end{array}\frac{aq^{q}}{4\eta^{\eta}p^{p}}
                                                                                              \right],~~~~~~~~~\eqno(3.2)$$

where $\nu,\mu,\gamma,\delta,\lambda \in\mathbb C$, $\Re(\nu)>0, \Re(\lambda)>0, \Re(\gamma)>0, \Re(\delta)>0$; $p,q>0$ and $q<\Re(\alpha)+p$.

\vspace{.35cm}
\noindent {\bf (iii).} On setting $ p=\delta=1$ and replacing $\nu$ by $ \nu-1 $ in (2.1) and then by using (1.7), we find:
$$\int_{0}^{1} \int_{0}^{1} y^{\lambda}~(1-x)^{\lambda-1}~(1-y)^{\mu-1}~(1-xy)^{1-\lambda-\mu}~E_{\nu,q}^{\eta,\gamma} \left[\frac{ay(1-x)(1-y)}{(1-xy)^{2}}\right]~{dx}~{dy}~~~~~~~~~~~~~$$
$$=\frac{1}{\Gamma(\gamma)}~~{}_3\Psi_{2}\left[\begin{array}{ccccc}
                                                 (\gamma,~q), & (\lambda,~1), & (\mu,~1); &  \\
                                                  &  &  &  & ~a \\
                                                 (\eta,\nu), & (\lambda+\mu,~2)  ; &
                                               \end{array}\right],~~~~~~\eqno(3.3)$$

where $ \nu,\mu,\gamma,\lambda,\eta\in\mathbb C$; $\Re(\nu)>0,\Re(\mu)>0,Re(\gamma)>0, Re(\delta)>0$ and $q\in(0,1)\bigcup N$.

\vspace{.35cm}
\noindent {\bf (iv).} On setting $ p=\delta=1 $ and replacing $\nu$ by$ \nu-1 $ in (2.4) and then by using (1.7), we attain:
$$\int_{0}^{1} \int_{0}^{1} y^{\lambda}~(1-x)^{\lambda-1}~(1-y)^{\mu-1}~(1-xy)^{1-\lambda-\mu}~E_{\nu,q}^{\eta,\gamma} \left[\frac{ay(1-x)(1-y)}{(1-xy)^{2}}\right]~{dx}~{dy}~~~~~~~~~~~~~$$
$$=~\frac{\Gamma(\lambda)\Gamma(\mu)}{\Gamma(\nu)\Gamma(\lambda+\mu)}~{}_{q+2}F_{\eta+2}\left[
                                                                                                \begin{array}{ccccc}
                                                                                                  ~\Delta(q;\gamma), & ~\lambda ,& ~\mu ,& ~ & ; \\
                                                                                                  ~ & ~ & ~ & ~ & ~ \\
                                                                                                  ~\Delta(\eta;\nu), & ~ & ~\Delta(2;\lambda+\mu), & ~ & ; \\
                                                                                                \end{array}\frac{aq^{q}}{4\eta^{\eta}}
                                                                                              \right],~~~~~~~~~\eqno(3.4)$$

where $ \nu,\mu,\gamma,\lambda,\eta\in\mathbb C$; $\Re(\nu)>0,\Re(\mu)>0,Re(\gamma)>0, Re(\delta)>0$ and $q\in(0,1)\bigcup N$.

\vspace{.35cm}
\noindent {\bf (v).} On setting $p=q=\delta=1$ and replacing $\nu$ by $ \nu-1$ in (2.1) and then by using (1.6), we achieve:
$$\int_{0}^{1} \int_{0}^{1} y^{\lambda}~(1-x)^{\lambda-1}~(1-y)^{\mu-1}~(1-xy)^{1-\lambda-\mu}~E_{\eta,\nu}^{\gamma} \left[\frac{ay(1-x)(1-y)}{(1-xy)^{2}}\right]~{dx}~{dy}~~~~~~~~~~~~~$$
$$=\frac{1}{\Gamma(\gamma)}~~{}_3\Psi_{2}\left[\begin{array}{ccccc}
                                                 (\gamma,~1), & (\lambda,~1), & (\mu,~1); &  \\
                                                  &  &  &  & ~a \\
                                                 (\eta,\nu), & (\lambda+\mu,~2)  ; &
                                               \end{array}\right],~~~~~~\eqno(3.5)$$

where $\nu,\eta,\gamma,\lambda,\mu\in\mathbb C$; $\Re(\nu)>0,\Re(\lambda)>0,Re(\gamma)>0,\Re(\mu)>0$.

\vspace{.35cm}
\noindent {\bf (vi).} On setting $ p=q=\delta=1 $ and replacing $\nu$ by $ \nu-1 $ in (2.4) and then by using (1.6), we find:
$$\int_{0}^{1} \int_{0}^{1} y^{\lambda}~(1-x)^{\lambda-1}~(1-y)^{\mu-1}~(1-xy)^{1-\lambda-\mu}~E_{\eta,\nu}^{\gamma} \left[\frac{ay(1-x)(1-y)}{(1-xy)^{2}}\right]~{dx}~{dy}~~~~~~~~~~~~~$$
$$=~\frac{\Gamma(\lambda)\Gamma(\mu)}{\Gamma(\nu)\Gamma(\lambda+\mu)}~{}_{3}F_{\eta+2}\left[
                                                                                                \begin{array}{ccccc}
                                                                                                  ~\Delta(1;\gamma), & ~\lambda ,& ~\mu ,& ~ & ; \\
                                                                                                  ~ & ~ & ~ & ~ & ~ \\
                                                                                                  ~\Delta(\eta;\nu), & ~ & ~\Delta(2;\lambda+\mu), & ~ & ; \\
                                                                                                \end{array}\frac{a}{4\eta^{\eta}}
                                                                                              \right],~~~~~~~~~\eqno(3.6)$$

where $\nu,\eta,\gamma,\lambda,\mu\in\mathbb C$; $\Re(\nu)>0,\Re(\lambda)>0,Re(\gamma)>0,\Re(\mu)>0$.

\vspace{.35cm}
\noindent {\bf (vii).} On setting $ p=q=\delta=\gamma=1$ and replacing $\nu$ by $ \nu-1 $ in (2.1) and then by using (1.5), we get:
$$\int_{0}^{1} \int_{0}^{1} y^{\lambda}~(1-x)^{\lambda-1}~(1-y)^{\mu-1}~(1-xy)^{1-\lambda-\mu}~E_{\eta,\nu} \left[\frac{ay(1-x)(1-y)}{(1-xy)^{2}}\right]~{dx}~{dy}~~~~~~~~~~~~~$$
$$={}_3\Psi_{2}\left[\begin{array}{ccccc}
                                                 (1,~1), & (\lambda,~1), & (\mu,~1); &  \\
                                                  &  &  &  & ~a \\
                                                 (\eta,\nu), & (\lambda+\mu,~2)  ; &
                                               \end{array}\right],~~~~~~\eqno(3.7)$$

where $ \nu,\eta,\mu,\lambda\in\mathbb C$; $\Re(\nu)>0,Re(\mu)>0, Re(\lambda)>0$.

\vspace{.35cm}

\noindent {\bf (viii).} On setting $ p=q=\delta=\gamma=1$ and replacing $\nu$ by $ \nu-1 $ in (2.4) and then by using (1.5), we obtain:
$$\int_{0}^{1} \int_{0}^{1} y^{\lambda}~(1-x)^{\lambda-1}~(1-y)^{\mu-1}~(1-xy)^{1-\lambda-\mu}~E_{\eta,\nu} \left[\frac{ay(1-x)(1-y)}{(1-xy)^{2}}\right]~{dx}~{dy}~~~~~~~~~~~~~$$
$$=~\frac{\Gamma(\lambda)\Gamma(\mu)}{\Gamma(\nu)\Gamma(\lambda+\mu)}~{}_{2}F_{\eta+2}\left[
                                                                                                \begin{array}{ccccc}
                                                                                                  ~ & ~\lambda ,& ~\mu ,& ~ & ; \\
                                                                                                  ~ & ~ & ~ & ~ & ~ \\
                                                                                                  ~\Delta(\eta;\nu), & ~ & ~\Delta(2;\lambda+\mu), & ~ & ; \\
                                                                                                \end{array}\frac{a}{4\eta^{\eta}}
                                                                                              \right],~~~~~~~~~\eqno(3.8)$$

where $ \nu,\eta,\mu,\lambda\in\mathbb C$; $\Re(\nu)>0,Re(\mu)>0, Re(\lambda)>0$.

 \vspace{.35cm}

\noindent {\bf(ix).} On setting $p=q=\delta=\gamma=1$ and $\nu=0$ in (2.1) and then by using (1.4), we attain:
$$\int_{0}^{1} \int_{0}^{1} y^{\lambda}~(1-x)^{\lambda-1}~(1-y)^{\mu-1}~(1-xy)^{1-\lambda-\mu}~E_{\eta} \left[\frac{ay(1-x)(1-y)}{(1-xy)^{2}}\right]~{dx}~{dy}~~~~~~~~~~~~~$$
$$={}_3\Psi_{2}\left[\begin{array}{ccccc}
                                                 (1,~1), & (\lambda,~1), & (\mu,~1); &  \\
                                                  &  &  &  & ~a \\
                                                 (\eta,\nu), & (\lambda+\mu,~2)  ; &
                                               \end{array}\right],~~~~~~\eqno(3.9)$$
where $ \mu,\lambda,\eta\in\mathbb C$; $\Re(\mu)>0,\Re(\lambda)>0$.

 \vspace{.35cm}

 \noindent {\bf(x).} On setting $ p=q=\delta=\gamma=1 $ and $ \nu=0 $ in (2.4) and then by using  (1.4), we get:
$$\int_{0}^{1} \int_{0}^{1} y^{\lambda}~(1-x)^{\lambda-1}~(1-y)^{\mu-1}~(1-xy)^{1-\lambda-\mu}~E_{\eta} \left[\frac{ay(1-x)(1-y)}{(1-xy)^{2}}\right]~{dx}~{dy}~~~~~~~~~~~~~$$
$$=~\frac{\Gamma(\lambda)\Gamma(\mu)}{\Gamma(\lambda+\mu)}~{}_{2}F_{\eta+2}\left[
                                                                                                \begin{array}{ccccc}
                                                                                                  ~ & ~\lambda ,& ~\mu ,& ~ & ; \\
                                                                                                  ~ & ~ & ~ & ~ & ~ \\
                                                                                                  ~\Delta(\eta;\nu), & ~ & ~\Delta(2;\lambda+\mu), & ~ & ; \\
                                                                                                \end{array}\frac{a}{4\eta^{\eta}}
                                                                                              \right],~~~~~~~~~\eqno(3.10)$$

where $ \mu,\lambda,\eta\in\mathbb C$; $\Re(\mu)>0,\Re(\lambda)>0$.

\end{document}